\documentclass[12pt]{article} 

\usepackage[margin=2.5cm]{geometry}
\usepackage[english]{babel}
\usepackage{hyperref}
\usepackage{amsmath}
\usepackage{amsthm}
\usepackage{amssymb}
\usepackage{ucs}
\usepackage{enumerate}
\usepackage[utf8x]{inputenc}
\usepackage[T1]{fontenc}

\DeclareMathOperator{\supp}{\mathrm{supp}}

\DeclareMathOperator{\RR}{\mathcal{R}}

\newcommand{\MAlg}{\mathrm{MAlg}}
\newcommand{\Aut}{\mathrm{Aut}}

  \newcommand{\R}{\mathbb R}
  \newcommand{\N}{\mathbb N}
  \newcommand{\Q}{\mathbb Q}
  \newcommand{\Z}{\mathbb Z}
  \newcommand{\C}{\mathbb C}
  \newcommand{\LL}{\mathrm L}

  \newcommand{\inv}{^{-1}}

  \renewcommand{\ker}{\mathrm{Ker}\,}

  \renewcommand{\leq}{\leqslant}
  \renewcommand{\geq}{\geqslant}
  
  \newcommand{\act}{\curvearrowright}
  
  \newcommand{\impl}{\Rightarrow}

  \newcommand{\into}{\hookrightarrow}
  
  \DeclareMathOperator{\normal}{\triangleleft}

\newtheorem{thmi}{Theorem}

\newtheorem{thm}{Theorem}[section]
\newtheorem{cor}[thm]{Corollary}
\newtheorem{proper}[thm]{Properties}
\newtheorem{lem}[thm]{Lemma}
\newtheorem{prop}[thm]{Proposition}

\theoremstyle{definition}
\newtheorem*{claim}{Claim}

\newtheorem{qu}[thm]{Question}

\newtheorem{df}[thm]{Definition}
\newtheorem*{rmq}{Remark}
\newtheorem{exemple}[thm]{Example}

\renewcommand{\d}{d}
\newcommand{\ph}{\varphi}
\newcommand{\eps}{\varepsilon}
\newcommand{\gtild}{\widetilde{[\mathcal R_G]}}

\title{More Polish full groups}

\author{Alessandro Carderi and François Le Maître}
\date{}
\begin{document}

\maketitle

\begin{abstract}

We associate to every action of a Polish group on a standard probability space a Polish group that we call the \textit{orbit full group}. For discrete groups, we recover the well-known full groups of pmp equivalence relations equipped with the uniform topology. However, there are many new examples, such as orbit full groups associated to measure preserving actions of locally compact groups. In fact, we show that such full groups are complete invariants of orbit equivalence. 

We give various characterizations of the existence of a dense conjugacy class for orbit full groups,
 and we show that the ergodic ones actually have a unique Polish group topology. Furthermore, we characterize ergodic full groups of countable pmp equivalence relations as those admitting non-trivial continuous character representations. 
 \end{abstract}

\tableofcontents

\section{Introduction}

One of the main problems in ergodic theory is to understand measure preserving actions of a countable group $\Gamma$ on the standard probability space $(X,\mu)$ up to conjugacy. This problem turns out to be very complicated in many different ways, even for $\Gamma=\Z$. In this case, we are trying to understand the conjugacy classes of the group of automorphisms of the standard probability space, $\Aut(X,\mu)$. 

The notion of orbit equivalence is much coarser than conjugacy. Two actions of countable groups $\Gamma$ and $\Lambda$ on the standard probability space $(X,\mu)$ are \textbf{orbit equivalent} if there exists $S \in\Aut(X,\mu)$ such that for almost every $x\in X$, we have $S(\Gamma\cdot x)=\Lambda\cdot S(x)$. In other words, the orbit equivalence only keeps track of the partition of the space into orbits induced by $\Gamma$ and $\Lambda$ and not of the actions themselves. A remarkable example is the following result of Ornstein and Weiss: any two measure-preserving ergodic actions of any two amenable groups are orbit equivalent \cite{MR551753}.

Orbit equivalence may be reformulated as follows. To any action of a countable group $\Gamma$ on the standard probability space $(X,\mu)$, we can associate the equivalence relation $\RR_\Gamma$ on $X$ defined by $(x,y)\in \RR_\Gamma$ if there exists $\gamma\in \Gamma$ such that $\gamma x=y$. Then two actions are orbit equivalent if and only if their equivalence relations are isomorphic up to measure zero. The equivalence relations arising in this way are called countable pmp (probability measure preserving) equivalence relations. They have geometric and cohomogical interpretations as well as fruitful relations with von Neumann algebras. We refer the interested reader to the survey of Gaboriau \cite{MR2827853}. 

Another way of formulating orbit equivalence is due to Dye. Suppose that $\Gamma$ acts on the standard probability space $(X,\mu)$. We define the \textbf{full group} induced by the $\Gamma$-action, to be the set of all $T\in\Aut(X,\mu)$ such that for almost every $x\in X$, we have $T(x)\in \Gamma\cdot x$. This group still encodes orbit equivalence in the following sense: two actions are orbit equivalent if and only if their full groups are conjugate in $\Aut(X,\mu)$. Moreover a theorem of Dye (see Theorem \ref{thm:dyerec}) implies that full groups are complete invariants of orbit equivalence, meaning that two actions are orbit equivalent if and only if their full groups are abstractly isomorphic.

As a consequence, one should be able to understand all the invariants of orbit equivalence in terms of full groups. This works well for ergodicity: an action is ergodic if and only if the associated full group is a simple group. 

In order to understand finer orbit equivalence invariants in terms of properties of the full group, it is better to introduce a Polish group topology on them. This topology is called the \textbf{uniform topology}, and it is induced by the \textbf{uniform metric} $d_u$ defined on $\Aut(X,\mu)$ by \[\d_u(T,S):=\mu\left(\left\lbrace x\in X:\ Tx\neq Sx\right\rbrace\right).\]
For example, Giordano and Pestov proved in \cite{MR2311665} that if $\Gamma$ acts freely on $(X,\mu)$ then it is amenable if and only if its full group is extremely amenable for the uniform topology. Another example is given by the topological rank. The topological rank of a topological group is the minimal number of elements needed to generate a dense subgroup. It was shown by the second named author that the topological rank of a full group can be expressed in terms of a fundamental invariant of orbit equivalence: the cost \cite{gentopergo}.\\

In this work we generalize the notion of full groups to actions of arbitrary Polish groups. Given a measure-preserving action of the Polish group $G$ on the standard probability space $(X,\mu)$, we define the \textbf{orbit full group} of the action exactly as before: it is the set of $T\in \Aut(X,\mu)$ such that for almost every $x\in X$, we have $T(x)\in G\cdot x$. We will denote this full group by $[\RR_G]$ to remember that it is the full group of the equivalence relation induced by the $G$-action. We should warn the reader that our definition needs a concrete action of $G$ on $X$, and not just a morphism $G\to \Aut(X,\mu)$.

As we said before, in order to understand deeper orbit full groups, we have to introduce a Polish topology on them. All the orbit full groups are closed for the uniform topology, but they are separable if and only if they arise as full groups of countable pmp equivalence relations. This does not rule out the existence of a Polish topology on them, for instance a compact group acting on itself by translation generates the transitive equivalence relation, so the associated orbit full group is $\Aut(X,\mu)$, which is a Polish group for the weak topology. 

The aim of this work is to define a Polish group topology on \textit{all} orbit full groups. We call it the topology of convergence in measure. It is not in general the restriction of a topology on $\Aut(X,\mu)$. When the action of $G$ on $X$ is free, one can define this topology as follows. We associate to an element $T\in[\RR_G]$ the function $f:X\to G$ uniquely defined by $T(x)=f(x)\cdot x$. Doing so, we embed $[\RR_G]$ in the space of measurable functions from $X$ to $G$, and our Polish topology is really the topology induced by the topology of convergence in measure.

\begin{thmi}[Thm. \ref{thm:polishfullgroups} and Thm. \ref{thm:atmostonepolishergo}]\label{thm:polishfullgroupsintro}
Let $G$ be a Polish group acting in a measure-preserving Borel manner on a standard probability space $(X,\mu)$. Then the associated orbit full group
$$[\mathcal R_G]=\{T\in\Aut(X,\mu): \forall x\in X, T(x)\in G\cdot x\}$$
is a Polish group for the topology of convergence in measure. 

Moreover, if the $G$-action is ergodic, then $[\mathcal R_G]$ has a unique Polish group topology, and if the $G$-action is free, then $G$ embeds into $[\mathcal R_G]$. 
\end{thmi}

The topologies we get with this theorem are not always the uniform or the weak topology. Indeed, any Polish locally compact group $G$ admits a free ergodic measure-preserving action, see for example Proposition 1.2 of \cite{MR1250814}. If the group $G$ is not countable, then the topology of convergence in measure on $[\mathcal R_G]$ cannot be the uniform topology, because $G\subset\Aut(X,\mu)$ is discrete for the uniform topology. Moreover if $G$ is not compact, then we show in Corollary \ref{cor:noteqaut} that $[\RR_G]\neq \Aut(X,\mu)$, so this topology is not the weak topology either, by Corollary \ref{cor:polishforweak}. 

Dye in \cite{MR0131516} gave an abstract definition of full groups: a subgroup $\mathbb G\leq \Aut(X,\mu)$ is \textbf{full} if for every countable $\Gamma\leq \mathbb G$, the full group generated by $\Gamma$ is still a subgroup of $G$. Clearly the orbit full groups we have defined are full groups in the sense of Dye's definition. 

We do not know how exactly which full groups have a Polish group topology. Nevertheless, we show that if an ergodic full group admits a Polish topology, then such topology is unique, refines the weak topology and is weaker than the uniform topology (Theorem \ref{thm:atmostonepolishergo}). It follows that if an ergodic full group admits a Polish topology, then it is a Borel subset of $\Aut(X,\mu)$ (Corollary \ref{cor:ergopolishborel}). This allows us to give examples of full groups which cannot carry a Polish group topology (Corollary \ref{cor:nopol}). Note that such a phenomenon is actually common in the  topological setting, as was recently shown by Ibarlucias and Melleray \cite{Ibarlucia:2013kq}.\\

One of the main interests of full groups induced by actions of countable groups is that they are complete invariants of orbit equivalence. For Polish group actions, the notion of orbit equivalence is slightly more complicated. Two actions of two Polish groups $G$ and $H$ on the standard probability space $(X,\mu)$ are orbit equivalent if there exists a full measure subset $A\subseteq X$ and a measure preserving bijection $S: A\to A$ such that for all $x\in A$,
\[S(G\cdot x)\cap A=(H\cdot S(x))\cap A.\]
It is clear from this definition that if two actions are orbit equivalent, then their orbit full groups are conjugate in $\Aut(X,\mu)$. The converse is however more complicated. Suppose that two ergodic orbit full groups are isomorphic. Dye's Reconstruction Theorem (see Theorem \ref{thm:dyerec}) still holds, so the isomorphism is the conjugation by some $S\in\Aut(X,\mu)$. When both acting groups are locally compact, we show that $S$ induces an orbit equivalence. So orbit full groups are complete invariants of orbit equivalence for actions of locally compact groups (Theorem \ref{thm:reconstlc}).

The orbit full groups that we have defined arise as intermediate examples between full groups of countable pmp equivalence relations and $\Aut(X,\mu)$, so they should share the topological properties which are satisfied by both. One of the simplest of such properties is contractibility, and indeed it is not hard to see that orbit full groups are contractible using the same the same approach of Keane for $\Aut(X,\mu)$ in \cite{MR0265555} (see Corollary \ref{cor:ofgcontra}).

However $\Aut(X,\mu)$ and full groups of countable pmp equivalence relations have many opposite properties. For example, any aperiodic\footnote{$T\in\Aut(X,\mu)$ is \textbf{aperiodic} if all its orbits are infinite.} element has a dense conjugacy class in $\Aut(X,\mu)$, while in the full group of a countable pmp equivalence relation, the identity cannot be approximated by aperiodic elements. We now characterize when a group action induces an orbit full group for which the aperiodic elements have a dense conjugacy class. 

\begin{thmi}[Thm. \ref{thm:aperdense}] \label{thm:aperdenseintro}Let $\mathbb G$ be an orbit full group. Then the following are equivalent:
\begin{enumerate}[(i)]
\item the set of aperiodic elements is dense in $\mathbb G$;
\item the conjugacy class of any aperiodic element of $\mathbb G$ is dense in $\mathbb G$;
\item whenever $\Gamma\act(X,\mu)$ is a free measure-preserving action of a countable discrete group $\Gamma$;  there is a dense $G_\delta$ in $\mathbb G$ of elements inducing a free action of $\Gamma*\Z$;
\item for all neighborhood of the identity $U$ in $G$, the set of $x\in X$ such that $U\cdot x\neq \{x\}$ has full measure.
\end{enumerate}
\end{thmi}
Note that condition $(iii)$ is inspired by results that Törnquist obtained for $\mathbb G=\Aut(X,\mu)$ \cite{MR2210067}. 
Using condition $(iv)$, we get a nice dichotomy for measure-preserving ergodic actions of  locally compact groups: either they generate a countable pmp equivalence relation, or all the above conditions are satisfied (see Corollary \ref{cor:dicholc}). 
\\

In the last section we classify character representations of ergodic orbit full groups. Before stating the theorem, let us give some background. 

Every unitary representation of a group $G$ splits as direct sum of cyclic representations. These representations are encoded by the positive type functions, that are the functions $f:G\to \C$ such that for all finite tuple $(g_1,...,g_n)$ of elements of $G$, the matrix $(f(g_ig_j\inv))_{i,j=1,...,n}$ is positive semi-definite.

A positive type function $\chi:G\to\C$ is a \textbf{character} if it satisfies the following conditions:\begin{itemize}\item it is conjugacy-invariant: for all $g,h\in G$, we have  $\chi(g\inv h g)=\chi(h)$  and \item it is normalized: $\chi(1_G)=1$.\end{itemize} A \textbf{character representation} is a unitary representation of $G$ which splits as a direct sum of cyclic representations whose corresponding positive definite functions are characters. Character representations are actually the representations into the unitary groups of \textit{finite von Neumann algebras}, see \cite[Sec. 2.3]{MR3019724} for more details.

Every discrete group $\Gamma$ has a faithful character representation, namely the regular representation. It is associated to the \textit{regular} character $\chi_r$ defined by $\chi_r(\gamma)=0$ if $\gamma\neq 1_\Gamma$ and $\chi_r(1_\Gamma)=1$. The set of characters of $\Gamma$ is convex and compact for the pointwise topology. Moreover, it is a \textit{Choquet simplex}, meaning that every character can be written in a unique way as an integral of extremal characters. The problem of classifying extremal characters has a long history, starting with the work of Thoma who classified extremal characters of the group of permutations of the integers with finite support \cite{MR0173169}. Since then, many examples were studied, see for instance \cite{Peterson:2013jk} and references therein. 

The set of continuous characters of a locally compact group is again a Choquet simplex, but locally compact groups do not necessarily have a faithful character representation. For example, all the continuous character representations of connected semi-simple Lie groups are trivial by a result of Segal and von Neumann \cite{MR0037309}. Recently Creutz and Peterson have shown that the same is true for non discrete totally disconnected simple locally compact groups having the Howe-Moore property \cite[Thm. 4.2]{Creutz:2013jk}.

For Polish groups, the situation is more complicated. The set of continuous characters may cease to form a Choquet simplex. For example, the abelian group of measurable maps into the circle $G=\LL^0(X,\mu,\mathbb S^1)$ has no continuous extremal characters, although it has continuous characters (see \cite[Ex. C.5.10]{MR2415834}). However, if the Polish group $G$ contains a countable dense subgroup $\Gamma$ which has only countably many extremal characters, then the continuous extremal characters of $G$ are given by the extremal characters of $\Gamma$ which extend continuously to $G$. It is then easy to see that the continuous characters of $G$ form a Choquet subsimplex of the characters of $\Gamma$. This remark has been crucial for the understanding of continuous characters of several Polish groups. In particular, it was used by Dudko to give a complete description of the characters of the full group of the hyperfinite ergodic equivalence relation \cite{MR2813475}. We extend his result and classify all the characters of an arbitrary ergodic orbit full group which are continuous for the uniform topology. 

 \begin{thmi}[Thm. \ref{thm:descrichar}]\label{thm:descricharintro}Let $\mathbb G$ be an ergodic orbit full group. Then we have the following dichotomy:
\begin{enumerate}
\item Either $\mathbb G$ is the full group of a countable pmp equivalence relation, and then all its continuous characters are (possibly infinite) convex combinations of the characters $\chi_k$ given by 
$$\chi_k(g):=\mu(\{x\in X: g\cdot x=x\})^k$$
for $k\in\N$ and the constant character $\chi_0\equiv 1$.
\item  or $\mathbb G$ does not have any nontrivial continuous character representation. 
\end{enumerate}
\end{thmi}


\section{Preliminaries}

\subsection{Polish spaces}

A \textbf{Polish space} is a topological space which is separable and admits a compatible complete metric. A countable intersection of open subsets of a topological spaces is called a $G_\delta$. 

\begin{prop}[{\cite[Thm. 3.11]{MR1321597}}]\label{prop:polishsub}
Let $(X,\tau)$ be a Polish space. Then $Y\subset X$ is Polish for the induced topology if and only if $Y$ is a $G_\delta$.
\end{prop}

A \textbf{standard Borel space} is an uncountable set $X$ equipped with a $\sigma$-algebra $\mathfrak B$ such that there exists a Polish topology on $X$ for which $\mathfrak B$ is the $\sigma$-algebra of Borel subsets. A fundamental fact is that all the standard Borel spaces are isomorphic \cite[Thm. 15.6]{MR1321597}, and that every uncountable Borel subset of a standard Borel space is a standard Borel space when equipped with the induced $\sigma$-algebra \cite[Cor. 13.4]{MR1321597}. Moreover, the injective Borel image of a Borel set is Borel:

\begin{thm}[{Luzin-Suslin, see \cite[Thm. 15.1]{MR1321597}}]\label{thm:luzinsuslin}Let $X$ and $Y$ be two standard Borel spaces and let $f: X\to Y$ be an injective Borel map. Then for every Borel subset $A$ of $X$, $f(A)$ is Borel.
\end{thm}

A subset $A$ of a Polish space $X$ is \textbf{analytic} if there is a standard Borel space $Y$, a Borel subset $B$ of $Y$ and a Borel function $f: Y\to X$ such that $A=f(B)$. In general, analytic sets are not Borel, however they are Lebesgue-measurable (see Thm. 4.3.1  in \cite{MR1619545}). If $X$ and $Y$ are two Polish spaces, a map $f:X\to Y$ will be called analytic if the preimage of any open set is analytic. Note that analytic maps are Lebesgue-measurable by the aforementioned result.
\subsection{Polish groups}

 A topological group whose topology is Polish is called a \textbf{Polish group}. Polish groups have several good properties which we now list, for proofs see Section 1.2 of \cite{MR1425877}.

\begin{proper}\label{proper}\hfill
  \begin{description}
  \item[($\alpha$)] Let $G$ be a Polish group, and let $H$ be a subgroup of $G$. Then $H$ is Polish for the induced topology if and only if $H$ is closed in $G$.
  \item[($\beta$)]  Let $G$ be a Polish group, and let $H\normal G$ be a closed normal subgroup. Then $G/H$ is a Polish group for the quotient topology.
  \item[($\gamma$)] Let $\ph: G \to H$ be an analytic homomorphism between two Polish groups $G$ and $H$. Then $\ph$ is continuous. If moreover $\ph$ is surjective, then $\ph$ induces a topological isomorphism between $G/\ker(\ph)$ and $H$. 
  \end{description}
\end{proper}

Let us end this section by citing a deep result of Becker and Kechris, which will be crucial in the proof of our main theorem. 

\begin{thm}[\cite{MR1185149}]\label{thm:borelvscont}
Let a Polish group $G$ act in a Borel manner on a standard Borel space $X$. Then there exists a Polish topology $\tau$ on $X$ inducing its Borel structure such that the action of $G$ on $(X,\tau)$ is continuous. 
\end{thm}

\subsection{The Polish group $\Aut(X,\mu)$}

A \textbf{standard probability space} is a standard Borel space equipped with a non atomic probability measure. All standard probability spaces are isomorphic (see  \cite[Thm. 17.41]{MR1321597}). The \textbf{measure algebra} of the standard probability space $(X,\mu)$ is the $\sigma$-algebra of measurable subsets of $X$, where two such subsets are identified if their symmetric difference has measure zero. We will denote the measure algebra with $\MAlg(X,\mu)$ and recall that it is a Polish space when equipped with the topology induced by the complete metric $d_\bigtriangleup$ defined by $d_\bigtriangleup(A,B)=\mu(A\bigtriangleup B)$.

\begin{df}\label{df:topoaut}
  Let $(X,\mu)$ be a standard probability space. The group $\Aut(X,\mu)$ of measure-preserving Borel bijections of $(X,\mu)$, identified up to measure zero, carries two natural metrizable topologies :
  \begin{itemize}
  \item the \textbf{weak topology}, for which a sequence $(T_n)_n$ converges to $T$ if for every measurable subset $A\subset X$, we have $\mu(T_n(A)\Delta T(A))\rightarrow 0$. 
  \item the \textbf{uniform topology},  induced by the \textbf{uniform metric} $\d_u$ defined by \[\d_u(T,S):=\mu\left(\left\lbrace x\in X:\ Tx\neq Sx\right\rbrace\right).\]
  \end{itemize}
\end{df}
\begin{prop}[{\cite{MR0111817}}]
  The group $\Aut(X,\mu)$ is a Polish group with respect to the weak topology, and the uniform metric $\d_u$ is complete. 
\end{prop}

\subsection{Spaces of measurable maps}\label{sec:lzero}

\begin{df}\label{df:ll0}
Let $(X,\mu)$ be a standard probability space, and let $(Y,\tau)$ be a Polish space. We denote by $\LL^0(X,\mu,(Y,\tau))$ the space of Lebesgue-measurable maps from $X$ to $Y$, identified up to measure 0.  Any compatible bounded metric $\d$ on $(Y,\tau)$ induces a metric $\tilde \d$ on $\LL^0(X,\mu,(Y,\tau))$ defined by
$$\tilde \d(f,g):=\int_X\d(f(x),g(x))\d\mu(x).$$
The topology induced by $\tilde \d$ is called the topology of \textbf{convergence in measure}.
\end{df}

This topology only depends on the topology of $Y$ by the following proposition.

\begin{prop}[{\cite[Prop. 6]{MR0414775}}]\label{prop:equivcvmeasure}Let $(f_n)$ be a sequence of elements of $\LL^0(X,\mu,(Y,\tau))$ and $f\in \LL^0(X,\mu,(Y,\tau))$. Then the following are equivalent:
\begin{enumerate}[(a)]
\item $\tilde \d(f_n,f)\to 0$,
\item \label{item:criterneighb}for all $\eps>0$, $\mu(\{x\in X: \d(f(x),f_n(x))>\eps\})\to 0$,
\item \label{item:critertopo}every subsequence of $(f_n)_{n\in\N}$ admits a subsequence $(f_{n_k})_{k\in\N}$ such that for almost all $x\in X$ we have $$f_{n_k}(x)\to f(x).$$
\end{enumerate}
\end{prop}
\begin{rmq}
In a topological space, a sequence converges to a point if and only if all its subsequences have a subsequence converging to this point, so item $(c)$ of the previous proposition implies that the topology of convergence in measure is the coarsest metrizable topology $\tau$ for which $f_n\to f$ a.s. implies that $f_n\to f$ with respect to $\tau$. 
\end{rmq}

The topology of convergence in measure on $\LL^0(X,\mu,Y)$ is a Polish topology. A dense countable subset is  constructed as follows. Fix a dense countable subset $D\subset Y$ and a dense countable subalgebra $\mathcal A$ of $\MAlg(X,\mu)$. Then, the family of $\mathcal A$-measurable functions from $X$ to $D$ that take only finitely many values is dense in $\LL^0(X,\mu,Y)$. 

When $Y=G$ is a Polish group, we equip $\LL^0(X,\mu,G)$ with the group structure given by the pointwise product, that is for $f,g\in\LL^0(X,\mu,G)$ we put $$f\cdot g(x):=f(x)g(x).$$

\begin{prop} The Polish space \label{prop:polishpointwise} $\LL^0(X,\mu,G)$ is a Polish group for the topology of convergence in measure and the pointwise product.
\end{prop}

Let $(Y,\tau)$ be a Polish space. Then $\Aut(X,\mu)$ acts on the right on $\LL^0(X,\mu,(X,\tau))$ by pre-composition: if $T\in \Aut(X,\mu)$ and $f\in \LL^0(X,\mu,(Y,\tau))$, then we define \linebreak $(f\cdot T)(x):=f(T x)$. Note that this is an action by isometries. Moreover, when $Y=X$, we may view $\Aut(X,\mu)$ as a subset of $\LL^0(X,\mu,(X,\tau))$ identifying a transformation $T$ with the function $f_T(x)=T(x)$, which corresponds to identifying $\Aut(X,\mu)$ with the orbit of $\mathrm{id}_X\in\LL^0(X,\mu,(X,\tau)$.

\begin{prop} \label{prop:continuity}
  Let $(X,\mu)$ be a standard probability space equipped with a compatible Polish topology $\tau_X$, and let $(Y,\tau_Y)$ be a Polish space.
  \begin{enumerate}[(1)]
    \item The action of $\Aut(X,\mu)$ on $\LL^0(X,\mu,(Y,\tau_Y))$ is continuous.
  \item The inclusion $\Aut(X,\mu)\into \LL^0(X,\mu,(X,\tau_X))$ is an embedding.
  \end{enumerate}
\end{prop}
\begin{proof}
$(1)$. 
Fix a compatible complete bounded metric $\d_Y$ on $Y$. Suppose now that $T_n\rightarrow T$ and $f_n\rightarrow f$, we want to prove that $\tilde\d_Y(f_nT_n,fT)\rightarrow 0$. Since each $T_n$ is an isometry, \[\tilde\d_Y(f_nT_n,fT)=\tilde\d_Y(f_n, fTT_n\inv)\leq \tilde\d_Y(f_n,f)+\tilde \d_Y(f, fTT_n\inv),\]
hence it is enough to show that if $T_n\rightarrow \mathrm{id}_X$ in $\Aut(X,\mu)$, then for every $f\in\LL^0(X,\mu,(Y,\tau))$ we have $fT_n\to f$ in measure.
Moreover we can suppose that $f$ has finite range, because the set of such functions is dense. For such a function $f$, we can consider the finite partition of the space given by $\{f^{-1}(a)\}_{a\in f(X)}$ and by definition of weak convergence $\mu(T_n f^{-1}(a)\Delta f^{-1}(a))\rightarrow 0$ for every $a\in f(X)$. So $\mu(\{x\in X: fT_n(x)\neq f(x)\})\to 0$, in particular $fT_n\to f$ in measure.

 $(2)$. 
Since $\Aut(X,\mu)$ can be identified with the orbit of $\mathrm{id}_X$ in $\LL^0(X,\mu,X)$, $(1)$ implies that the inclusion is continuous. 
 
To see that it is an embedding,  we first note that since $\Aut(X,\mu)$ acts by homeomorphisms on $\LL^0(X,\mu,\tau)$, a sequence $(T_n)_n$ converges to $T$ in measure if and only if $T_nT\inv\to\mathrm{id}_X$ in measure. So it is enough to show that every weak neighborhood of the identity contains a neighborhood of the identity for the topology of convergence in measure.

Given $r>0$ and a subset $A$ of $X$, we let $(A)_{r}:=\{y\in X: \exists x\in A, \d(x,y)<r\}$ be its $r$-neighborhood. Since the measure $\mu$ is regular, a pre-basis of neighborhoods of the identity for the weak topology on $\Aut(X,\mu)$ is given by the sets 
$$\mathcal V_{F,\eps}=\{T\in\Aut(X,\mu): \mu(F\bigtriangleup T(F))<\eps\},$$ where $F$ is closed and $\eps$ is positive.  
So fix such a closed set $F$ and $\eps>0$. Since $F$ is closed, we have that $F=\bigcap_{n\in\N} (F)_{\frac 1n}$, so there exists $\delta>0$ with $\delta<\eps$, such that $\mu( (F)_\delta\setminus F)<\eps$.

Now suppose that $\tilde d_X(T,\mathrm{id}_X)\leq \delta^2$, which implies that $\mu\left(\left\lbrace x\in X:\ \d_X(Tx,x)>\delta\right\rbrace\right)<\delta$. Then $\mu(T(F)\setminus (F)_\delta)<\delta<\eps$ and since $T$ preserves the measure, \[\mu(T(F)\Delta F)\leq \mu(T(F)\Delta (F)_\delta)+\mu((F)_\delta\Delta F)\leq 3\eps.\]

We conclude that $T\in\mathcal V_{F,3\epsilon}$, and so the topology of  convergence in measure refines the weak topology, which ends the proof.
\end{proof}

Combining the above proposition and Proposition \ref{prop:polishsub}, we deduce that $\Aut(X,\mu)$ is a $G_\delta$ in $\LL^0(X,\mu,(X,\tau))$. However, it is not closed, and one can actually show that its closure consists in the monoid of all measure preserving maps $(X,\mu)\to(X,\mu)$. 

\section{Full groups}

\subsection{Definition and fundamental facts}

Let us start by recalling the original definition of full groups introduced by Dye in his pioneering work \cite{MR0131516}, which is the starting point of our paper. 

\begin{df}
Let $(T_n)_{n\in\N}$ be a sequence of elements of $\Aut(X,\mu)$. We say that $T\in\Aut(X,\mu)$ is obtained by \textbf{cutting and pasting} the sequence $(T_n)_n$ if there exists a partition $(A_n)_{n\in\N}$ of $X$ such that for every $n\in\N$,
$$T_{\restriction A_n}=T_{n\restriction A_n}.$$
We will also say that $T$ is obtained by cutting and pasting  $(T_n)_{n\in\N}$ along $(A_n)_{n\in\N}$.
\end{df}

\begin{df}[Dye]\label{df:fullgroup}
A subgroup $\mathbb G$ of $\Aut(X,\mu)$ is a \textbf{full group} if it is stable under cutting and pasting any sequence of elements of $\mathbb G$. 
\end{df}

Before defining the main examples of full groups, let us point out a fundamental fact.

\begin{prop}[{\cite[Lem. 5.4]{MR0131516}}]\label{prop:dyeunif}
The restriction of the uniform metric $d_u$ to any full group is complete.
\end{prop}

\begin{df}
Suppose $\mathcal R$ is an equivalence relation on a standard probability space $(X,\mu)$. The \textbf{full group of the equivalence relation $\mathcal R$}, denoted by $[\mathcal R]$, is the set of all $T\in\Aut(X,\mu)$ such that  for all  $x\in X$,  $(x,T(x))\in\mathcal R.$
\end{df}
\begin{rmq}In the previous definition, we require that $(x,T(x))\in\mathcal R$ for \textit{all} $x\in X$, but up to modifying $T$ on a measure zero set we could as well ask that this holds for \textit{almost} all $x\in X$. Indeed, if $T$ satisfies the latter condition, let $A$ be the full measure set of $x\in X$ such that $(x,T(x))\in \mathcal R$. Then $A$ contains the full measure $T$-invariant set $B:=\bigcap_{n\in\Z} T^n(A)$ and $T$ coincides up to measure zero with the bijection $T'$ defined by 
$$T'(x):=\left\{\begin{array}{ll}T(x) & \text{ if }x\in B, \\x & \text{ else.}\end{array}\right.$$
It is then clear that for all $x\in X$, we have $(x,T'(x))\in \mathcal R$.
\end{rmq}


A special and important case of the previous definition is when the equivalence relation is given by the action of a group. 

\begin{df}
Let $G$ be a Polish group acting in a Borel manner\footnote{This just means that the action map $G\times X\to X$ is Borel. We do not assume here that the action is measure preserving.} on a standard probability space $(X,\mu)$. The associated \textbf{orbit full group} is the set of all  
\[T\in \Aut(X,\mu)\text{ such that }T(x)\in G\cdot x,\text{ for all }x\in X.\]
In other words, it is the full group of the orbit equivalence relation $\mathcal R_G$ defined by 
$ (x,y)\in\mathcal R_G \text{ whenever }\exists g\in G: g\cdot x=y,$
and we will accordingly denote it by $[\mathcal R_G]$.
\end{df}

We want to stress that the previous definition makes sense only for \textit{spatial} $G$-actions: we need a genuine $G$-action on $X$ in order to define $[\mathcal R_G]$, and not just a morphism $G\to \Aut(X,\mu)$. The orbit full group should not be confused with the following one.

\begin{df}
 Let $G$ be a subgroup of $\Aut(X,\mu)$. There exists a smallest full group containing $G$, whose elements are obtained by cutting and pasting elements of $G$. This is the \textbf{full group generated} by $G$, denoted by $[G]_D$.
\end{df}

If $G$ acts faithfully on $(X,\mu)$ and preserves the probability measure $\mu$, then we have an injective map $G\hookrightarrow \Aut(X,\mu)$ and we clearly have $[G]_D\subset [\mathcal R_G]$. The inclusion is in general strict, as shown in Example \ref{ex:dyestrictsub} and Example \ref{ex:kolmogorovex}. 

Let us now give some concrete examples. We start with full groups of countable measure preserving equivalence relations, which are exactly the countably generated full groups.

\begin{exemple} 
Suppose $\Gamma$ is a countable group acting on the standard probability space $(X,\mu)$ by measure preserving automorphisms . Then one can easily check that the two full groups defined above coincide: $[\Gamma]_D=[\mathcal R_\Gamma]$. Moreover, the orbit full group $[\mathcal R_\Gamma]$ is separable (see \cite[Prop. 3.2]{MR2583950})  and hence it is Polish with respect to the uniform topology by Proposition \ref{prop:dyeunif}  (this follows also from the proof of Theorem \ref{thm:polishfullgroups}). 
\end{exemple}

As a matter of fact, the full groups given by the previous example are the only full groups that are Polish for the uniform topology. 

\begin{prop}\label{prop:polishfordu}
  A full group $\mathbb G<\Aut(X,\mu)$ is Polish with respect to the uniform topology if and only if it is the full group of a countable probability measure preserving equivalence relation. 
\end{prop}
\begin{proof}
If $\mathbb G\leq\Aut(X,\mu)$ is a full group separable for the uniform topology, we can choose a dense countable subgroup $\Lambda\subset \mathbb G$ so that $[\Lambda]$ is a closed subgroup of $\mathbb G$ that contains $\Lambda$, hence it has to be equal to $\mathbb G$. On the other hand, any countable pmp equivalence relation comes from the action of a countable group by a result of Feldman and Moore \cite{MR0578656}, hence it is Polish for the uniform topology.
\end{proof}

\begin{exemple}\label{ex:dyestrictsub}
Let $G$ be an infinite compact group and denote its Haar measure by $\mu$. The  action of $G$ on itself by left translation generates the transitive equivalence relation on $G$, so its orbit full group is by definition $[\mathcal R_G]=\Aut(G,\mu)$. In particular, it is Polish for the weak topology. 

We remark now that the full group $[G]_D$ generated by $G$ may be strictly smaller than $[\mathcal R_G]$. Indeed, let us consider the circle group $G=\mathbb S^1$ acting on itself by translation, and suppose that the map $T: g\mapsto g\inv$ was obtained by cutting and pasting the translations by some $g_i$ along $A_i$. Fix an index $i$ and $h,h'\in A_i$ such that $h^2\neq {h'}^2$. Since $h\in A_i$, we have that $h\inv=g_ih$ and since $h'\in A_i$, we also get ${h'}\inv=g_ih'$, hence $h^2=g_i\inv={h'}^2$, a contradiction.
\end{exemple}

We say that a subgroup $G$ of $\Aut(X,\mu)$ is \textbf{ergodic} if for every $A\subset X$ such that $\mu(A\Delta gA)=0$ for every $g\in G$, we have that $A$ has either measure zero or full measure. This is equivalent to say that the only $G$-invariant elements of $\MAlg(X,\mu)$ are $X$ and $\emptyset$. Note that $G$ is ergodic if and only if $[G]_D$ is. Ergodic full groups have the following very useful property.
\begin{prop}\label{prop:transsamemeas}
Let $\mathbb G$ be an ergodic full group and let $A,B\in\MAlg(X,\mu)$ such that $\mu(A)=\mu(B)$. Then there exists an involution $T\in\mathbb G$ such that $T(A)=B$. 
\end{prop}

This proposition was used used by Fathi to show the following result (see \cite{MR0486410}; he proves this for $\Aut(X,\mu)$, but the same argument works for any ergodic full group). 

\begin{thm}A full group is ergodic if and only if it is simple.
\end{thm}
Proposition \ref{prop:transsamemeas} can also be used to show the following (see \cite[Prop. 3.1]{MR2583950}).

\begin{prop}\label{prop:weaklydenseerg}
  A full group is weakly dense in $\Aut(X,\mu)$ if and only if it is ergodic. 
\end{prop}

The following result is a consequence of Property \ref{proper} ($\alpha$) and the above proposition.
\begin{cor}\label{cor:polishforweak}
  The group $\Aut(X,\mu)$ is the unique ergodic full group that is Polish for the weak topology. 
\end{cor}

The previous corollary may be extended to the non ergodic case (see \cite[Prop. 1.17 and Thm. D.6]{lemai2014}), which yields a complete description of the full groups which are Polish for the weak topology.

Let us now give a pathological example.

\begin{exemple}\label{ex:kolmogorovex}
  Let $G=\mathfrak S_\infty$ be the Polish group of all permutations of the natural numbers. The action of $\mathfrak S_\infty$ on $X=\{0,1\}^{\mathbb N}$ is faithful, has only countably many orbits and all the orbits but one are countable.  Kolmogorov observed, see \cite[Ex. 9]{MR1760945}, that there exists a measure $\mu$ on $X$ for which the action of $G$ is not ergodic. This implies that $[G]_D$ is not ergodic and by Proposition \ref{prop:weaklydenseerg}, $[G]_D$ is not weakly dense in $[\mathcal R_G]=\Aut(X,\mu)$. 
\end{exemple}

\subsection{Orbit full groups have a Polish group topology}

So far, the examples of Polish full groups that we have seen are full groups of countable pmp equivalence relations, which are exactly the full groups  for which  the uniform topology is Polish by Proposition \ref{prop:polishfordu}, and $\Aut(X,\mu)$, which is the unique ergodic Polish full group for the weak topology by Corollary \ref{cor:polishforweak}. Note that both are instances of orbit full groups. Our main result is that \textit{all} orbit full groups carry a natural Polish group topology.\\

Before defining this topology, we need a better description of orbit full groups. So we start with a Polish group $G$ acting in a Borel manner on a standard probability space $(X,\mu)$. 

Consider the Polish space $\LL^0(X,\mu,G)$ endowed with the topology of convergence in measure, and define $\Phi: \LL^0(X,\mu,G)\to \LL^0(X,\mu,X)$ by $$\Phi(f)(x):=f(x)\cdot x.$$ 
In what follows, we will always see the group $\Aut(X,\mu)$ as a subspace of $\LL^0(X,\mu,X)$. Put $\gtild:=\Phi\inv(\Aut(X,\mu))$. 

\begin{lem}\label{lem:phisur}
We have the equality $\Phi\left(\gtild\right)= [\mathcal R_G]$. 
\end{lem}
\begin{proof}
The inclusion $\Phi\left(\gtild\right)\subseteq [\mathcal R_G]$ follows directly from the definition. For the reverse inclusion, take $T\in[\mathcal R_G]$ and consider the Borel set \[ A:=\left\lbrace (x,g)\in X\times G\ :\ gx=Tx\right\rbrace.\]

By the Jankov-von Neumann Uniformisation Theorem (see \cite[Thm. 18.1]{MR1321597}), we can find an analytic uniformisation of $A$, i.e. an analytic, hence Lebesgue-measurable, map $f:X\to G$ such that for every $x\in X$, we have $T(x)=f(x)x$. In other words, we can find $f\in \LL^0(X,\mu,G)$ such that $T= \Phi(f)$, so $[\mathcal R_G]\subseteq \Phi(\gtild)$.
\end{proof}

\begin{df}\label{df:topoofg}The \textbf{topology of convergence in measure} on an orbit full group $[\mathcal R_G]$ is the quotient topology induced by $\gtild\subseteq \LL^0(X,\mu,G)$, where we put on $\LL^0(X,\mu,G)$ the topology of convergence in measure (see Section \ref{sec:lzero}).
\end{df}

 We say that the action of $G$ is \textbf{essentially free} if there exists a full measure \linebreak$G$-invariant subset $A$ of $X$ such that for every $g\in G\setminus\{1_G\}$ and every $x\in A$, $gx\neq x$. Note that this is stronger than asking that all elements of $G\setminus\{1_G\}$ have a set of fixed points of measure zero, even when $G$ is locally compact.
 
Whenever the $G$ action is essentially free, $\Phi:\gtild\rightarrow [\mathcal R_G]$ is a bijection and the topology on the orbit full group is just the topology induced by the topology of convergence in measure on $\LL^0(X,\mu,G)$. We will give later a more precise description of the topology of convergence in measure on $[\mathcal R_G]$ when the action is non free (see Corollary \ref{cor:equivcvmeasure}). 

\begin{thm}\label{thm:polishfullgroups}
Let $G$ be a Polish group acting in a Borel manner on a standard probability space $(X,\mu)$. Then the associated orbit full group $$[\mathcal R_G]=\{T\in\Aut(X,\mu): \forall x\in X, T(x)\in G\cdot x\}$$ is a Polish group for the topology of convergence in measure. This topology is weaker than the uniform topology and refines the weak topology.

Moreover if the $G$-action is essentially free and measure preserving, then $G$ embeds into $[\mathcal R_G]$. 
\end{thm}

\begin{proof}
We start by showing that the topology of convergence in measure on $[\mathcal R_G]$  is a Polish group topology. 

By Theorem \ref{thm:borelvscont}, we may and do fix a Polish topology $\tau$ on $X$ such that $G\act(X,\tau)$ is a continuous action. Now, the characterization of the convergence in measure in terms of pointwise converging subsequences (cf. (\ref{item:critertopo}) in Proposition \ref{prop:equivcvmeasure})  yields that $\Phi$ is continuous.
Then, combining Proposition \ref{prop:continuity} and Proposition \ref{prop:polishsub}, we get that  $\Aut(X,\mu)\subseteq \LL^0(X,\mu,(X,\tau))$ is a $G_\delta$ and so  $\gtild=\Phi\inv(\Aut(X,\mu))$ is also a $G_\delta$. Therefore using again Proposition \ref{prop:polishsub} we see that $\gtild$ is Polish. 

We now equip $\gtild$ with the group operation $*$ defined by $$(f*g)(x):=f(\Phi(g)(x))g(x).$$
 The inverse is given by $f\inv(x):=f(\Phi(f)\inv x)\inv$. These group operations are continuous by Proposition \ref{prop:continuity}, Proposition \ref{prop:polishpointwise} and the fact that $\Phi$ is continuous. So $(\gtild,*)$ is a Polish group and it is easy to check that the restriction $\Phi_{\restriction\gtild}$ is a group homomorphism. Hence we deduce that \[[\mathcal R_G]=\Phi(\gtild)\cong\gtild/\ker(\Phi),\] is itself a Polish group for the quotient topology by Property \ref{proper} ($\beta$). 
\begin{rmq}
The Polish group $\gtild$ can be thought of as the full group of the groupoid associated to the action.
\end{rmq}

Let us now check the that the topology of convergence in measure is intermediate between the uniform and the weak topology. Since $\Phi$ is continuous, clearly we have that the topology on the orbit full group refines the weak topology, which also yields that $[\mathcal R_G]$ is a Borel subgroup of $\Aut(X,\mu)$ by Theorem \ref{thm:luzinsuslin}.

If the action of $G$ is essentially free, then $\Phi(f_n)\to \mathrm{id}_X$ uniformly implies that $\mu(f_n^{-1}(1_G))\to 1$ and hence the topology of convergence in measure is weaker than the uniform topology. The proof for non-free actions follows the same lines, but it will be an even more direct consequence of the description of the quotient topology that we will give in Proposition \ref{prop:dRG}.

Finally, when the $G$-action is essentially free, $\Phi$ restricts to a topological isomorphism between $\widetilde{[\mathcal R_G]}$ and $[\mathcal R_G]$. The ‘‘moreover'' part of the theorem then follows from the fact that $G$ embeds into $\gtild\subseteq\LL^0(X,\mu,G)$ by identifying $G$ with the set of constant maps.
\end{proof}
\begin{rmq}
We point out that whenever the orbit full group is ergodic, the topology we have defined is the unique possible Polish group topology, as we will show in Theorem \ref{thm:atmostonepolishergo}. 
\end{rmq}

\subsubsection*{Topology for non-free actions}

Let $G$ be a Polish group acting on $(X,\mu)$ as in Theorem \ref{thm:atmostonepolishergo}.

Recall that every Polish group admits a compatible right-invariant metric (see for instance \cite[Thm. 2.1.1]{MR2455198}). The proof of the following proposition can be found in \cite[Lem. 2.2.8]{MR2455198}.

\begin{prop}Let $G$ be a Polish group, and let $d_G$ be a compatible right-invariant metric on $G$. Suppose that $H\leq G$ is a closed subgroup, then the \textbf{quotient metric} $d_{G/H}$ on $G/H$ defined by, for $gH, g'H\in G/H$,
$$d_{G/H}(gH,g'H):=\inf_{h\in H} d_G(gh,g')$$
induces the quotient topology on $G/H$. 
\end{prop}

For $x\in X$, let $G_x:=\mathrm{Stab}_G(x)$. Then $G_x$ is a closed subgroup of $G$ by a result of Miller (see \cite[Thm. 9.17]{MR1321597}; this also follows from Theorem \ref{thm:borelvscont}). We now prove an analogous statement of the previous proposition for orbit full groups. 

\begin{prop}\label{prop:dRG}Let $G$ be a Polish group acting in a Borel manner on a standard probability space $(X,\mu)$. Let $d_G$ be a compatible bounded right-invariant metric on $G$. Then the quotient metric $d_{[\mathcal R_G]}$ induced by $\tilde d_G$ on $[\mathcal R_G]=\gtild/\ker\Phi$ is given by
$$d_{[\mathcal R_G]}(T,U)=\int_Xd_x(T(x),U(x))d\mu(x),$$
where for all $x\in X$,  we identify the $G$-orbit of $x$ to the homogenous space $G/G_x$, equipped with the quotient metric $d_x$ defined by $d_x(gG_x,g'G_x):=\inf_{h\in G_x}d_G(gh,g')$.
\end{prop}

Before proving the proposition, we state the following corollary. Its proof is analogous to the proof of Proposition \ref{prop:equivcvmeasure}, hence we omit it.

\begin{cor}\label{cor:equivcvmeasure}Let $(T_n)$ be a sequence of elements of an orbit full group $[\mathcal R_G]$, and let $T\in[\mathcal R_G]$. Then the following are equivalent:
\begin{enumerate}[(a)]
\item $T_n\to T$ in measure,
\item \label{item:criterneighbfg}for all $\eps>0$, $\mu(\{x\in X: d_x(T(x),T_n(x))>\eps\})\to 0$,
\item \label{item:critertopofg}every subsequence of $(T_n)_{n\in\N}$ admits a subsequence $(T_{n_k})_{k\in\N}$ such that for almost all $x\in X$ we have $T_{n_k}(x)\to T(x)$, where the convergence holds in the orbit of $x$, identified to the homogeneous space $G/G_x$.
\end{enumerate}
\end{cor}

\begin{proof}[Proof of Proposition \ref{prop:dRG}]
Let $K$ be the subgroup of $\LL^0(X,\mu,G)$ consisting of all $f: X\to G$ such that for all $x\in X$, $f(x)\in G_x$. It is clear that $K$ is the kernel of the restriction of $\Phi$ to $(\widetilde{[\mathcal R_G]},*)$. Moreover, for all $f\in\widetilde{[\mathcal R_G]}$ and $g\in K$,

$$(f*g)(x)= f(\Phi(g)(x))g(x)=f(x)g(x).$$
So two elements of $\gtild$ are in the same right $K$-coset for the group operation $*$ in $\widetilde{[\mathcal R_G]}$  if and only if they are in the same right $K$-coset with respect to the pointwise multiplication. This implies that the quotient metric induced by $\tilde d_G$ on  $\widetilde{[\mathcal R_G]}/K$ and the quotient metric induced by $\tilde d_G$ on $\LL^0(X,\mu,G)/K$ agree on $[\mathcal R_G]$. 

The latter metric comes from the pseudo-metric $\rho$ on $\LL^0(X,\mu,G)$ defined by $\rho(g,g'):=\inf_{k\in K}\tilde d_G(gk,g')$. In order to establish the proposition, we need to show that
$$\rho(g,g')=\int_Xd_x(g(x),g'(x))d\mu(x).$$
Note that the integral is well-defined because the function $(x,g,g')\mapsto d_x(g,g')$ is analytic, hence Lebesgue-measurable.  Fix $g,g'\in \LL^0(X,\mu,G)$. For every $\epsilon>0$, we apply the Jankov-von Neumann Uniformisation Theorem (see \cite[Thm. 18.1]{MR1321597}) to the set $$\{(x,h)\in X\times G: d_G(g(x)h,g'(x)<d_x(g(x),g'(x))+\eps\text{ and }h\cdot x=x\}$$ and we get a function $k\in K$ such that for all $x\in X$, 
$$d_G\left(g(x)k(x),g'(x)\right)<d_x(g(x),g'(x))+\eps.$$
This implies that $\rho(g,g')\leq \int_Xd_x(g(x),g'(x))d\mu(x)$. The reverse inequality is a direct consequence of the fact that for all  $k\in K$ and all $x\in X$, \[d_G(g(x)k(x),g'(x))\geq d_x(g(x),g'(x)).\qedhere\]
\end{proof}

\subsection{Full groups of measurable equivalence relations}

In this section, we study full groups of measurable equivalence relations, which will lead us to a simple criterion for distinguishing the orbit full groups of the previous section from $\Aut(X,\mu)$ (see Corollary \ref{cor:noteqaut}). 

\begin{prop}Let $(X,\mu)$ be a standard probability space, and $A$ be a Borel subset of $X\times X$. Then the set 
$$[A]:=\{T\in\Aut(X,\mu): \forall x\in X, (x,T(x))\in A\}$$
is a Borel subset of $\Aut(X,\mu)$. 
\end{prop}
\begin{proof}
We may suppose that $(X,d)$ is a Cantor set. By Proposition \ref{prop:continuity}, the weak topology on $\Aut(X,\mu)$ is the same as the topology of convergence in measure induced by $\LL^0(X,\mu,(X,d))$. We will show that given a Borel subset $A$ of $X\times X$, the function $\Phi_A:\LL^0(X,\mu,(X,d))\to \MAlg(X,\mu)$ defined by $$\Phi_A(f):=\{x\in X: (x,f(x))\in A\}$$
is Borel.  Let $\mathcal F$ be the class of subsets of $X\times X$ for which $\Phi_A$ is Borel. Because $X\times X$ is again a Cantor set, we only need to show that $\mathcal F$ is a $\sigma$-algebra containing the clopen sets.

So suppose that $A$ is a clopen subset of $X\times X$. We will actually show that $\Phi_A$ is continuous. 

Fix $f\in\LL^0(X,\mu,(X,d))$ and $\eps>0$. Since $A$ and $X^2\setminus A$ are open, there exists $\delta>0$ smaller than $\eps/2$ and a set $B$ of measure less than $\eps/2$ such that for all $x\in X\setminus B$ and $y\in X$ such that $\d(y,f(x))\leq \delta$, we have 
$$(x,f(x))\in A\Leftrightarrow (x,y)\in A.$$
It is then clear that for every $g$ in the open neighborhood $U$ of $f$ given by 
$$\mathcal U_\delta(f)=\left\lbrace g\in\LL^0(X,\mu,X): \mu\left\lbrace x\in X: \d(f(x),g(x))>\delta\rbrace\right)<\delta\right\rbrace,$$
we have $\d_\mu(\Phi_A(f),\Phi_A(g))<\eps$. 

So the class $\mathcal F$ contains the clopen sets, and we now have to show that it is a $\sigma$-algebra. First,  $\mathcal F$ is stable under complementation, since given $T\in\Aut(X,\mu)$, we have that $\Phi_{X^2\setminus A}(T)=X\setminus\Phi_A(T)$. Moreover if $(A_n)$ is a countable family of elements of $\mathcal F$, then 
$$\Phi_{\cup_n A_n}(T)=\bigcup_{n\in\N}\Phi_{A_n}(T),$$
and since taking a countable union is a Borel operation on $\MAlg(X,\mu)$, we get that $\bigcup_{n\in\N}A_n$ belongs to $\mathcal F$.
\end{proof}

\begin{cor}
Let $\mathcal R$ be a Borel equivalence relation on a standard probability space $(X,\mu)$. Then its full group
$$[\mathcal R]=\{T\in\Aut(X,\mu): \forall x\in X, (x,T(x))\in\mathcal R\}$$
is a Borel group for the Borel structure induced by the weak topology on $\Aut(X,\mu)$. 
\end{cor}

In the previous section, we saw that the orbit full group of a Borel action of Polish group is also Borel. Note however that there are actions of Polish groups inducing analytic, non Borel equivalence relations, so there are non Borel analytic equivalence relations whose full group is Borel. Using Theorem \ref{thm:borelfullgroups}, one can show that there also are analytic equivalence relations whose full group is not Borel.

Let us now give a general criterion which allows us to distinguish full groups of measurable equivalence relations and $\Aut(X,\mu)$.

\begin{prop}\label{prop:eqautxmu}
Let $\mathcal R$ be a Lebesgue-measurable equivalence relation on a standard probability space. Then the following are equivalent:
\begin{enumerate}[(i)]
\item$[\mathcal R]=\Aut(X,\mu)$,
\item $\mathcal R$ has full measure in $X\times X$,
\item $\mathcal R$ has an equivalence class of full measure.
\end{enumerate}
\end{prop}
\begin{proof}
Both implications $(iii)\impl(i)$ and $(iii)\impl(ii)$ are obvious. 

The implication $(ii)\impl (iii)$ is also easy: suppose that $\mathcal R$ has full measure, then by Fubini's Theorem for almost all $x\in X$, the $\mathcal R$-equivalence class of $x$ has full measure, so that in particular there is one equivalence class of full measure. 

For the remaining implication $(i)\impl(iii)$, assume that $[\mathcal R]=\Aut(X,\mu)$. We may suppose that $X$ is a compact group equipped with the Haar measure. By assumption, for all $z\in X$ the left multiplication by $z$ is in the full group $[\mathcal R]$,  which means that for all $z\in X$ and almost $x\in X$, $(zx,x)\in\mathcal R$. Again by Fubini's theorem, we deduce that there exists $x\in X$ such that for almost all $z\in X$, $(zx,x)\in\mathcal R$, hence $(iii)$ holds.
\end{proof}

\begin{rmq}
By the same arguments used in the previous proof, one can show that the full group generate by the circle group acting on itself by left translation $[\mathbb S^1]_D$ cannot be the full group of any Lebesgue-measurable equivalence relation. Indeed, such an equivalence relation would have to be transitive by the previous proof, so that $[\mathbb S^1]_D$ would be equal to $\Aut(X,\mu)$, contradicting Example \ref{ex:dyestrictsub}.
\end{rmq}

\begin{cor}\label{cor:noteqaut}Suppose a Polish group $G$ acts essentially freely on $(X,\mu)$, and that $[\mathcal R_G]=\Aut(X,\mu)$. Then $G$ is compact.
\end{cor}
\begin{proof}
If $[\mathcal R_G]=\Aut(X,\mu)$, then by the previous proposition there is a $G$-orbit of full measure. The freeness of the $G$-action allow us to identify such an orbit to $G$, which then carries a left-invariant Borel probability measure, hence $G$ is compact by Ulam's Theorem (see \cite[Thm. B.1]{MR2191233}). 
\end{proof}

\begin{rmq}If we moreover assume that $G$ is uncountable, $[\mathcal R_G]$ cannot be the full group of a countable pmp equivalence relation, because $G\leq[\mathcal R_G]$ is discrete for the uniform topology whenever the action is essentially free. So Theorem \ref{thm:polishfullgroups} yields a large class of new examples of Polish full groups whose Polish topology is neither the weak nor the uniform topology. 
\end{rmq}

\subsection{Orbit equivalence and full groups}

Let $\mathcal R$ and $\mathcal R'$ be two equivalence relations on a standard probability spaces $(X,\mu)$. We say that $\mathcal R$ and $\mathcal R'$ are \textbf{orbit equivalent} if there exist a full measure subset $A\subseteq X$ and a Borel measure preserving bijection $S: A\to A$ such that for all $(x,y)\in A^2$, $(x,y)\in\mathcal R$ if and only if $(S(x),S(y))\in \mathcal R'$. We also say that such an $S$ is an \textbf{orbit equivalence} between $\mathcal R$ and $\mathcal R'$. It is easy to see that if $S$ is an orbit equivalence, then it conjugates the full groups $[\mathcal R]$ and $[\mathcal R']$, that is, we have the relation $S[\mathcal R]S\inv=[\mathcal R']$. In the case of orbit full groups, one can say a bit more.

\begin{lem}Let $G$ and $H$ be two Polish groups acting in a Borel  manner on $(X,\mu)$. Let $S\in\Aut(X,\mu)$ is an orbit equivalence between $\mathcal R_{G}$ and $\mathcal R_{H}$. Then the conjugation by $S$ is a group homeomorphism between the orbit full groups $[\mathcal R_{G}]$ and $[\mathcal R_{H}]$. 
\end{lem}
\begin{proof}
By Theorem \ref{thm:polishfullgroups}, $[\mathcal R_{G}]$ and $[\mathcal R_{H}]$ are Borel subgroups of $\Aut(X,\mu)$ and the conjugation by $S$ induces a Borel isomorphism between $[\mathcal R_{G}]$ and $[\mathcal R_{H}]$. This isomorphism is a homeomorphism by Property \ref{proper} ($\gamma$).
\end{proof}

A fancier restatement of the previous lemma is that orbit full groups, seen as topological groups, are invariants of orbit equivalence. Now, we show that for ergodic measure preserving actions of locally compact groups, the orbit full groups are complete invariants. 

\begin{thm}\label{thm:reconstlc}Let $G$ and $H$ be two Polish locally compact groups acting in a Borel measure preserving ergodic manner on a standard probability space $(X,\mu)$.  Suppose that $\psi: [\mathcal R_G]\to [\mathcal R_H]$ is an abstract group isomorphism. Then there exists an orbit equivalence $S$ between $\mathcal R_G$ and $\mathcal R_H$ such that for all $T\in[\mathcal R_G]$, 
$$ \psi(T)=S\inv TS.$$
\end{thm}

Dye's Reconstruction Theorem plays a fundamental role in our proof.

\begin{thm}[{\cite[Thm. 2]{MR0158048}}]\label{thm:dyerec}
Suppose $\mathbb G_1$ and $\mathbb G_2$ are two ergodic full groups on a standard probability space $(X,\mu)$. Then for every abstract group isomorphism $\psi: \mathbb G_1\to \mathbb G_2$, there exists $S\in\Aut(X,\mu)$ such that for all $T\in \mathbb G_1$, we have $\psi(T)=STS\inv.$
\end{thm}

The theorem also uses the following proposition, whose proof is inspired by Proposition $B.2$ of \cite{MR776417}. 
  
\begin{prop}\label{prop:oeincl}
 Let $G$ and $H$ be two Polish locally compact groups acting in a Borel measure preserving manner on a standard probability space $(X,\mu)$, and suppose that $[\mathcal R_G]\subseteq [\mathcal R_H]$. Then there exists a full measure set $X_0\subseteq X$ such that $$\mathcal R_G\cap \left(X_0\times X_0\right)\subseteq \mathcal R_H.$$ 
\end{prop}
\begin{proof}
Let $\nu$ be the Haar measure on $G$. The fact that $[\mathcal R_G]\subseteq[\mathcal R_H]$ implies that for all $g\in G$ and almost all $x\in X$, we have $g x\in Hx$. By Fubini's Theorem, this implies that the set 
$$X_0:=\{x\in X: \text{ for }\nu\text{-almost all }g\in G,\text{ we have } gx\in Hx\}$$
has full measure. Now let $x\in X_0$, and let $g_1\in G$ be such that $g_1x\in X_0$. We want to show that $g_1x\in Hx$. 

Since $x$ and $g_1x$ are in $X_0$, the sets $A:=\{g\in G:gx\in Hx\}$ and $B:=\{g\in G:gx\in Hg_1x\}$ have full measure and so $Ag_1\inv \cap B$ has full measure. Take $g\in Ag_1\inv\cap B$, then $gg_1x\in Hx\cap Hg_1x$, so the two orbits $Hx$ and $Hg_1x$ intersect, hence $g_1x\in Hx$.
\end{proof}

\begin{proof}[Proof of Theorem \ref{thm:reconstlc}]
By Dye's Reconstruction Theorem, $\psi$ is the conjugation by some $S\in\Aut(X,\mu)$ and by the previous proposition applied two times, such an $S$ has to be an orbit equivalence. 
\end{proof}

\begin{qu}Can the previous theorem be extended to (some) non locally compact Polish groups?
\end{qu}

\section{Topological properties of full groups}

\subsection{Aperiodic elements and free actions}

Now we study some topological properties of orbit full groups. 
So let us fix a Polish group $G$ acting on a Borel manner on the standard probability space $(X,\mu)$. We will denote by $\mathbb G=[\mathcal R_G]$ the associated orbit full group. Recall that its Polish topology is weaker than the uniform topology and refines the weak topology. 

We will first prove that orbit full groups are contractible as a corollary of the continuity of the first return map.

\begin{df}
Let $T\in\Aut(X,\mu)$ and $A\in\MAlg(X,\mu)$, Poincaré's Recurrence Theorem states that for almost every $x\in A$ there is a smaller $n_x\in\N$ such that $T^{n_x}(x)\in A$. The \textbf{first return map} $T_A$ is then defined by $T_A(x)=x$ for all $x\not\in A$, and by $T_A(x)=T^{n_x}(x)$ for all $x\in A$.
\end{df}

Proposition \ref{prop:cofr} and Corollary \ref{cor:ofgcontra} are  straightforward generalizations of Keane's results for $\Aut(X,\mu)$ \cite{MR0265555}. We will however give a detailed proof to show how the compatible metric on orbit full groups defined in Proposition \ref{prop:dRG} can be used.

\begin{prop}\label{prop:cofr}Let $\mathbb G$ be an orbit full group. Then the function which maps $(T,A)\in\mathbb G\times \MAlg(X,\mu)$ to the first return map  $T_A\in\mathbb G$ is continuous.
\end{prop}

\begin{proof}
Given $T\in\Aut(X,\mu)$, $A\in\MAlg(X,\mu)$ and $n>0$, we let 
$$C_n(T,A):=\{x\in X: T^n(x)\in A\},$$
we put $B_n(T,A):=C_n(T,A)\setminus\bigcup_{m<n}C_m(T,A)$ and $B_0(T,A):=X\setminus A$. Then $(B_n(T,A))_{n\geq 0}$ is a partition of $X$ and $T_A(x)=T^n(x)$ for all $x\in B_n(T,A)$. Note that $B_n(T,A)$ depends continuously on $(T,A)\in\Aut(X,\mu)\times \MAlg(X,\mu)$, where $\Aut(X,\mu)$ is equipped with the weak topology. 

 Now, let $\eps>0$, and fix $(\tilde T,\tilde A)\in\mathbb G\times \MAlg(X,\mu)$. Since $(B_n(\tilde T, \tilde A))_{n\geq 0}$ is a partition of $X$, we may find $N>0$ such that $\mu(X\setminus\bigcup_{m<N} B_m(\tilde T,\tilde A))<\eps$. Let $\mathcal U$ be the set of couples $(T,A)\in \mathbb G\times \MAlg(X,\mu)$ such that 
 \begin{enumerate}[(1)]
\item  $\sum_{m=0}^N d_{[\mathcal R_G]}(\tilde T^m,T^m)<  \eps$ and 
\item $\sum_{m=0}^N \mu(B_m(T,A)\Delta B_n(\tilde T,\tilde A))< \eps$.
\end{enumerate}
By continuity of $T\mapsto T^m$ and of $B_m(\cdot,\cdot)$, the set $\mathcal U$ is open. Let $(T,A)\in\mathcal U$, we now compute the distance between $\tilde T_{\tilde A}$ and $T_A$. 
\begin{align*}d_{[\mathcal R_G]}(\tilde T_{\tilde A},T_A)=&\int_Xd_x(\tilde T_{\tilde A}(x),T_A(x))d\mu(x)\\
\leq &\ \sum_{m=0}^N\int_{B_m(\tilde T,\tilde A)\Delta B_m(T,A)}d_x(\tilde T_{\tilde A}(x),T_A(x))d\mu(x)\\ +&\sum_{m=0}^n\int_{B_m(\tilde T,\tilde A)\cap B_m(T,A)}d_x(\tilde T_{\tilde A}(x),T_A(x))d\mu(x) \\ +&\int_{\bigcup_{m>N}B_m( \tilde T,\tilde A)}d_x(\tilde T_{\tilde A}(x),T_A(x))d\mu(x).\end{align*}
Since the metric $d_x$ is bounded by $1$, we then get the following inequality:
\begin{align*}
d_{[\mathcal R_G]}(\tilde T_{\tilde A},T_A)&< \sum_{m=0}^N \int_{B_m(\tilde T,\tilde A)\cap B_m(T,A)}d_x(\tilde T_{\tilde A}(x),T_A(x))d\mu(x)+2 \eps \\
&<\sum_{m=0}^N \int_{B_m(\tilde T,\tilde A)\cap B_m(T,A)}d_x(\tilde T^m(x),T^n(x))d\mu(x)+2 \eps \\&< 3 \eps.\qedhere 
\end{align*}
\end{proof}

\begin{cor}\label{cor:ofgcontra}
Orbit full groups are contractible.
\end{cor}
\begin{proof}
We may suppose that $X=[0,1]$ and that $\mu$ is the Lebesgue measure. The continuous map 
$H: [0,1]\times \mathbb G\to \mathbb G$ defined by $H(s,T):=T_{[s,1]}$ is a homotopy between the identity function $T\mapsto T$ and the constant function $T\mapsto \mathrm{id}_X$. 
\end{proof}

The \textbf{support} of $T:X\to X$, denoted by $\supp T$, is defined by $$\supp T:=\{x\in X: T(x)\neq x\}.$$ We now state the main theorem of this section. Note that condition (v) is a  generalization of \cite[The Category Lemma]{MR2210067} to the case of orbit full groups.

\begin{thm}\label{thm:aperdense}Let $\mathbb G$ be an orbit full group. Then the following are equivalent:
\begin{enumerate}[(i)]
\item the set of aperiodic elements is dense in $\mathbb G$;
\item the conjugacy class of any aperiodic element of $\mathbb G$ is dense in $\mathbb G$;
\item there exists a sequence $(T_n)$ of aperiodic elements of $\mathbb G$ such that $T_n\to \mathrm{id}_X$;
\item for all $A\in\MAlg(X,\mu)$, there exists a sequence $(T_n)$ of elements of $\mathbb G$ such that $T_n\to\mathrm{id}_X$ and for all $n\in\N$, $A=\supp T_n$;
\item whenever $\Gamma\act(X,\mu)$ is a free measure-preserving action of a countable discrete group $\Gamma$,  there is a dense $G_\delta$ in $\mathbb G$ of elements inducing a free action of $\Gamma*\Z$; 
\item  for all $n\in\N$, there is a dense $G_\delta$ of $(T_1,...,T_n)$ in $\mathbb G^n$ which induce a free action of $\mathbb F_n$.
\item for all neighborhood of the identity $U$ in $G$, $\cup_{g\in U}\supp g$ has full measure.
\end{enumerate}
\end{thm}
\begin{proof}
Let us first prove that \textit{(i)} implies \textit{(ii)}, using the exact same argument as for $\Aut(X,\mu)$ (see \cite[Thm. 2.4]{MR2583950}). By Rokhlin's Lemma, the conjugacy class of any aperiodic element of $\mathbb G$ is dense in the set of aperiodic elements of $\mathbb G$ for the uniform topology (see \cite[Cor 5.11]{lemai2014} for details). Moreover the topology of $\mathbb G$ is weaker than the uniform topology, so \textit{(i)} implies that the conjugacy class of any aperiodic element is dense in $\mathbb G$, in other words \textit{(i)}$\impl$\textit{(ii)}.
Clearly \textit{(ii)} implies \textit{(iii)} and the continuity of the first-return map (Proposition \ref{prop:cofr}) yields that \textit{(iii)} implies \textit{(iv)}. 

Törnquist proved in \cite[The Category Lemma]{MR2210067} that \textit{(v)} holds for $\mathbb G=\Aut(X,\mu)$, and the only topological fact he used was precisely \textit{(iv)} so we can repeat his entire argument to get \textit{(iv)}$\impl$\textit{(v)} (see the observation before the proof of Lemma 2 in \cite{MR2210067}; the fact that $P$ is an involution is not relevant here). 

The implication \textit{(v)}$\impl$\textit{(vi)} is proven by induction, and \textit{(i)} is a reformulation of \textit{(vi)} for $n=1$, so the implication \textit{(vi)}$\impl$\textit{(i)} also holds.\\

So all the statements from \textit{(i)} to \textit{(vi)} are equivalent, and we now only have to prove that \textit{(vii)} is equivalent to them. For this, we will prove that \textit{(iii)} implies \textit{(vii)}, and then that \textit{(vii)} implies \textit{(iv)}.  

Let us show that \textit{(iii)} implies \textit{(vii)}. Assuming that \textit{(vii)} is not satisfied, we can find an open neighborhood of the identity $U$ in $G$ such that $\mu(\bigcup_{g\in U} \supp g)=1-\delta$ for some $\delta>0$. We define a neighborhood of the identity $\mathcal U$ in $\gtild$ by $$\mathcal U:=\{f\in \widetilde{[\mathcal R_G]}:\mu(\{x\in X: f(x)\not\in U\})<\delta/2\}.$$ For every $f\in\mathcal U$, we have $d_u(\Phi(f),\mathrm{id}_X)<1$, hence $\Phi(f)$ is not aperiodic. The projection of $\mathcal U$ on $[\mathcal R_G]$ is a neighborhood of the identity in $[\mathcal R_G]$ consisting of non-aperiodic elements, contradicting \textit{(iii)}. 

For the remaining implication \textit{(vii)}$\impl$\textit{(iv)},  we first use Theorem \ref{thm:borelvscont} and fix a topology $\tau$ on $X$ such that the action of $G$ on $(X,\tau)$ is continuous. 

Let $V$ be a neighborhood of the identity in $G$. We say that $T\in [\mathcal R_G]$ is \textit{uniformly small} if there exists $f\in \LL^0(X,\mu,G)$ such that $f(X)\subset V$ and $T(x)=f(x)x$ for every $x\in X$. Observe that if we prove that for every $A\in \MAlg(X,\mu)$ there exists a uniformly small $T\in [\mathcal R_G]$ such that $\supp(T)\subset A$, then we would have that for every $A\in \MAlg(X,\mu)$  we could construct a uniformly small $T\in [\mathcal R_G]$ such that $\supp(T)= A$ by a maximality argument. And so condition \textit{(iv)} would be  satisfied.

So let us fix an open set $B\supseteq A$ such that $\mu(B\setminus A)<\mu(A)$ and an open neighborhood of the identity $U\subseteq G$ such that $U\inv=U$ and $U^2\subseteq V$.

\begin{claim} There exists a countable family $(g_i)_{i\in\N}$ of elements of $U$ and an a.e. partition $(A_i)_{i\in\N}$ of $A$ such that for all $i\in\N$, $g_i(A_i)$ is a subset of $B$ which is disjoint from $A_i$. 
\end{claim}
\begin{proof}[Proof of the claim]
Let $(U_n)_{n\in\N}$ be a countable basis of open neighborhoods of the identity in $G$, and let $$S:=\bigcap_{n\in\N} \bigcup_{g\in U_n} \supp g.$$
By hypothesis, $S$ has full measure. Moreover since the action is continuous, for all $x\in S\cap B$ there exists a $g\in U$ such that $g\cdot x\in B$ and $g\cdot x\neq x$. Always by continuity we can also find an open neighborhood $W_x\subseteq B$ of $x$ such that $g(W_x)$ and $W_x$ are disjoint. We can now define the partition $\{A_n\}$ to be a countable open subcover of $(W_x)_{x\in S\cap B}$ which exists by Lindelöf's Theorem. 
\end{proof}

 We now have two possible cases. 
\begin{itemize}
\item If for some $i\in\N$, $g_i(A_i)$ is not disjoint from $A$, then we set $C:=A_i\cap g_i\inv(g_i(A_i)\cap A)$ and the element $T$ of $[\mathcal R_G]$ defined by 
$$T(x):=\left\{\begin{array}{ll}g_i\cdot x & \text{if }x\in C \\g_i\inv\cdot x & \text{if }x\in g(C) \\x & \text{otherwise}\end{array}\right.$$
is uniformly small, non trivial, and supported in $A$.
\item If for all $i\in\N$, $g_i(A_i)$ is disjoint from $A$, then since $\mu(B\setminus A)<\mu(A)$ and $\{A_i\}_i$ is a partition of $A$, there exist two distinct indices $i,j\in\N$ such that $g_i(A_i)\cap g_j(A_j)$ has positive measure. Letting $C:=g_i\inv(g_i(A_i)\cap g_j(A_j))$, we see that the element $T$ of $[\mathcal R_G]$ defined by 
$$T(x):=\left\{\begin{array}{ll}g_ig_j\cdot x & \text{if }x\in C \\g_jg_i\inv\cdot x & \text{if }x\in g_ig_j\inv(C) \\x & \text{otherwise}\end{array}\right.$$
is uniformly small, non trivial, and supported in $A$. \qedhere
\end{itemize}
\end{proof}
\begin{rmq}Condition $(vii)$ is always satisfied as soon as $G$ is non discrete and acts essentially freely and is never satisfied for countable discrete groups. In fact, it easy to see that the aperiodic elements form a closed proper subset of full groups of countable pmp equivalence relations.
\end{rmq}

\begin{cor}\label{cor:dicholc}
Let $G$ be a locally compact group acting ergodically and in a Borel measure-preserving manner on $(X,\mu)$. Then we have the following dichotomy:
\begin{enumerate}[(1)]
\item either $\mathcal R_G$ is a countable pmp equivalence relation,
\item or the set of aperiodic elements is dense in $[\mathcal R_G]$.
\end{enumerate}
\end{cor}
\begin{proof}
Let $\{U_n\}_n$ be a sequence of open neighborhoods of the identity in $G$ such that $\cap_n U_n=\{id\}$. We define the \textbf{core} of the action to be the intersection \[\bigcap_{n\in \N}\bigcup_{g\in U_n} \supp g.\]

Note that the core is a Borel set by Theorem \ref{thm:borelvscont}. Moreover the core is $G$-invariant because for all $g,h\in G$, $h(\supp g)=\supp(hgh\inv)$ and the conjugation by $h$ is a continuous isomorphism of $G$. By ergodicity, either the core has measure one, in which case (2) holds by a direct application of condition $(vii)$ of the previous theorem, or the core has measure zero.

So assume that the core has measure zero. Then it is easy to check that the $G$-action yields a uniformly continuous morphism $G\to (\Aut(X,\mu),d_u)$. By separability of the group $G$, we can conclude from Proposition \ref{prop:polishfordu} that $[G]_D$ is the full group of a countable pmp equivalence relation. Moreover $G$ is locally compact, so we can apply Proposition \ref{prop:oeincl} to deduce that the equivalence relation $\mathcal R_G$ is countable up to a measure zero set, that is $(1)$ holds. 
\end{proof}

Let us give a non-trivial example where condition (1) of the previous corollary is satisfied. 

\begin{exemple}
Let $(\Gamma_n)_{n\in\N}$ be a sequence of finite groups and let $G:=\prod_{n\in\N} \Gamma_n$ be their product. Let $(X,\mu)$ be a standard probability space, and fix a partition $(A_n)$ of $X$ such that each $A_n$ has positive measure.  For every $n\in\N$, we then fix  a measure-preserving  action $\alpha_n$ of the finite group $\Gamma_n$ which is free when restricted to $A_n$, and which is trivial outside of it.
  
  We can now define an action of the group $G$ on $X$ by $(\gamma_n)_n\cdot x=\alpha_m(\gamma_m)x$ whenever $x\in A_{m}$. This action is faithful and its core is trivial. Note also that we can embed in this way any profinite group into any ergodic full group $\mathbb G$. It is in fact sufficient to take the actions $\alpha_n$ such that $\alpha_n(G_n)\subset \mathbb G$. 
\end{exemple}


\subsection{Uniqueness of the Polish topology}

We now show that the topology of convergence in measure is the unique possible Polish group topology for ergodic orbit full groups.

\begin{thm}\label{thm:atmostonepolishergo}
Let $\mathbb G$ be an ergodic full group. 
\begin{enumerate}[(1)]
\item Any Polish group topology on $\mathbb G$ is weaker than the uniform topology .
\item Any Polish group topology on $\mathbb G$ refines the weak topology.
\item The group $\mathbb G$ carries at most one Polish group topology.
\end{enumerate}
\end{thm}

Item \textit{(1)} was shown to hold by Kittrell and Tsankov for the full group of an ergodic countable pmp equivalence relation \cite[Thm. 3.1]{MR2599891}, while item \textit{(2)} was proven by Kallman for $\mathbb G=\Aut(X,\mu)$ in \cite{MR796452}. Hence we will not give a detailed proof of these points, but we will just indicate how to adapt the previous results to our broader setting. 
\begin{proof}
\textit{(1)} The proof of Theorem $3.1$ in \cite{MR2599891} only makes use of ergodicity via Proposition \ref{prop:transsamemeas}, and of the fact that in any full group, every element may be written as the product of at most three involutions \cite{MR1244984}.
So it adapts verbatim to obtain that $(\mathbb G, d_u)$ has the automatic continuity property\footnote{A topological group $G$ has the \textbf{automatic continuity property} if every morphism from $G$ into a separable group $H$ is continuous. }. We deduce that the identity map $(\mathbb G,d_u)\to (\mathbb G,\tau)$ is continuous, in other words the uniform topology refines $\tau$.

\textit{(2)} The arguments in \cite{MR796452} also use only Proposition \ref{prop:transsamemeas} and they adapt verbatim to obtain that the subsets of the form $\{g\in\mathbb G: \mu(A\bigtriangleup g A)<\eps\}$ are analytic for any Polish group topology $\tau$ on $\mathbb G$. In particular, the identity map $(\mathbb G, \tau)\to (\Aut(X,\mu), w)$ is Baire-measurable, hence continuous by Property \ref{proper} ($\gamma$), so that $\tau$ refines $w$.

\textit{(3)} Let $\tau$ and $\tau'$ be two Polish group topologies on $\mathbb G$. By \textit{(2)}, both topologies have to refine the weak topology, so the inclusion map $\mathbb G\into (\Aut(X,\mu),w)$ is Borel for both topologies. By Theorem \ref{thm:luzinsuslin}, the Borel $\sigma$-algebras of $\mathbb G$ for $\tau$ and $\tau'$ are both equal to the Borel $\sigma$-algebra induced by the  weak topology. So $\tau=\tau'$ by Property \ref{proper} ($\gamma$).
\end{proof}

We note that the proof of \textit{(3)} yields the following corollary.

\begin{cor}\label{cor:ergopolishborel}Let $\mathbb G$ be an ergodic full group admitting a Polish topology. Then it is a Borel subset of $\Aut(X,\mu)$ for the weak topology.
\end{cor}

We do not know whether the full group generated by the circle acting on itself by translation is Polishable. Note however that this full group is Borel by Corollary \ref{cor:boreldye}.  The following question has a positive answer when $\mathbb G$ is either $\Aut(X,\mu)$ \cite{MR3042607}, or the full group of a countable pmp ergodic equivalence relation \cite{MR2599891}.

\begin{qu}Let $\mathbb G$ be an ergodic orbit full group, then does it have the automatic continuity property? 
\end{qu}

Let us now observe that the third point of Theorem \ref{thm:atmostonepolishergo} still holds in the case of a full group with a countable ergodic decomposition, using a simple result from the second named author's thesis, which we prove for the reader's convenience. 

\begin{lem}\label{lem:uniprod}
Let $(G_i)_{i\in\N}$ be a countable family of groups with trivial center admitting at most one Polish group topology. 
Then $\prod_{i\in\N}G_i$ has at most one Polish topology.
\end{lem}
\begin{proof}Fix a Polish group topology on $\prod_{i\in\N} G_i$, and for $i\in \N$ denote by $\pi_i$ the projection on $G_i$. Let $n\in\N$, set \[G'_n:=\{g\in \prod_{i\in\N}G_i:\forall k\neq n, \pi_k(g)=e\}.\]

 Because each $G_i$ has trivial center, $G'_n$ is the commutator of $H_n:=\{g\in\prod_{i\in\N}G_i: \pi_n(g)=e\}$, hence $G'_n$ is closed in $\prod_{i\in \N} G_i$. Since $G_n$ is isomorphic to $G'_n$, the induced topology on $G'_n$ is the unique Polish group topology on $G_n$. In particular, for every open subset $U$ of $G_n$, the set \[\tilde U:=\{e\}\times\cdots\times\{e\}\times U\times \{e\}\times\cdots \subset \prod_{i\in\N}G_i\] is a $G_\delta$.

Observe that $H_n$ is closed, since it is the commutator of $G'_n$. So for any open subset $U$ of $G_n$, the set \[\tilde U\cdot H_n=G_1\times\cdots\times G_{n-1}\times U\times G_{n+1}\times\cdots\] is analytic, which implies by Property \ref{proper} ($\gamma$) the uniqueness of the Polish topology of $\prod_{i\in\N} G_i$.
\end{proof}

A full group with countably many ergodic components is isomorphic to the product of the full groups of its restrictions to these ergodic components, so we can combine Theorem \ref{thm:atmostonepolishergo} \textit{(3)} with the previous lemma to get the following corollary. 

\begin{cor}Let $\mathbb G$ be a full group with countably many ergodic components. Then it carries at most one Polish group topology.
\end{cor}

\subsection{More (non) Borel full groups}

In this section, we give more examples of Borel full groups, as well as examples of non Borel ones. But first, we need some background on the space of probability measures.

Let $(Y,\tau)$ be a Polish space.  We equip the space $\mathcal P(Y)$  of probability measures on $Y$ with the weak-* topology, that is, the coarsest topology making the maps 
$$\mu\in\mathcal P(Y)\mapsto \int_Xfd\mu$$
continuous for all bounded continuous functions  $f: Y\to \R$. This is a Polish topology (see e.g. \cite[Sec. 17.E]{MR1321597}).
\begin{lem}\label{lem:pushfwdco}Let $(X,\mu)$ be a standard probability space, and $(Y,\tau)$ be a Polish space. Then the following map 
$$\left.\begin{array}{cl}\LL^0(X,\mu,Y) & \to \mathcal P(Y) \\f & \mapsto f_*\mu\end{array}\right.$$
is continuous.
\end{lem}
\begin{proof}
If a sequence $(f_n)_n$ converges to $f$ in measure, then for almost every $x\in X$, we have that $f_n(x)\to f(x)$. For every continuous and bounded function $g:Y\to \R$ by Lebesgue Dominated Convergence Theorem we have,
$$\int_Xgd{f_n}_*\mu=\int_X g(f_n(x)) d\mu(x) \to \int_Xg(f(x))d\mu(x)=\int_Xgd{f}_*\mu,$$
and hence the map is continuous. 
\end{proof}

\begin{prop} The set of completely atomic probability measures of a Polish space $Y$ is a Borel subset of $\mathcal P(Y)$.  
\end{prop}
\begin{proof}
We first note that whenever $U$ is an open subset of $Y$, then the set 
\begin{align*}  Atom(U):=\left\{\mu\in P(Y):\text{there exists }a\in U\text{ such that }\mu(U)=\mu(\{a\})\right\}
\end{align*}
is closed in $P(Y)$. Indeed if  $\mu\notin Atm(U)$, then there exist two positive functions $f$ and $g$ supported in $U$ with disjoint supports such that $\int_Xf d\mu$ and $\int_Xg d\mu$ are strictly positive.

Let $(U_n)_{n\in\N}$ be a countable basis of open subsets of $Y$. Let $\mathcal J$ be the set of finite subsets $I\subset \N$ such that $(U_i)_{i\in I}$ is a disjoint family of open sets. Since a measure $\mu$ is atomic if and only if for all $\epsilon>0$, there is an open set $U$ of $X$ of measure greater than $1-\epsilon$ containing finitely many atoms, the set of atomic measure is exactly \[\bigcap_{n\in \N} \bigcup_{I\in\mathcal  J}\left(\left\{\mu\in\mathcal P(Y): \mu\left(\bigcup_{i\in I} U_i\right)>1-\frac 1n\right\}\cap \bigcap_{i\in I} Atom(U_i)\right),\] and so it is Borel. 
\end{proof}

\begin{cor}\label{cor:ctblerange}Let $(X,\mu)$ be a standard probability space and $(Y,\tau)$ be a Polish space. Then the set $\LL^0_D(X,\mu,Y)$ of elements of $\LL^0(X,\mu,Y)$ with countable range is Borel.
\end{cor}
\begin{proof}
Observe that $f\in\LL^0(X,\mu,Y)$ has countable range if and only if $f_*\mu$ is completely atomic. So $\LL^0_D(X,\mu,Y)$ is the pre-image of a Borel set by a continuous function, hence it is Borel.   
\end{proof}

The next theorem is a Borel version of Theorem \ref{thm:polishfullgroups}. The proof in the Borel case is easier because we will not use a continuous model, that is, Theorem \ref{thm:borelvscont}. 

\begin{thm}\label{thm:borelfullgroups}
Let $G$ be a Polish group acting in a Borel way by measure preserving transformations on a standard probability space $(X,\mu)$. Suppose that the action is essentially free, and let $H\leq G$ be a subgroup of $G$. Then the following are equivalent:
\begin{enumerate}
\item $H$ is a Borel subgroup of $G$;
\item $[H]_D$ is a Borel subgroup of $\Aut(X,\mu)$;
\item $[\mathcal R_H]$ is a Borel subgroup of $\Aut(X,\mu)$.
\end{enumerate}
\end{thm}
\begin{proof}
As we did for the proof of Theorem \ref{thm:polishfullgroups}, we will consider the Polish space $\LL^0(X,\mu,G)$, and we use the fact that the map $\Phi: \LL^0(X,\mu,G)\to \LL^0(X,\mu,X)$ defined by 
$$\Phi(f)(x):=f(x)\cdot x$$
 is Borel\footnote{Note that $\LL^0(X,\mu,X)$ is a standard Borel space, whose Borel structure does not depend on the Polish topology we put on $X$ by \cite[Prop. 8]{MR0414775}.}. The homomorphism $\Phi$ is injective because the action is essentially free. 
Set $\widetilde{[H]_D}:=\Phi\inv([H]_D)$ and $\widetilde{[\mathcal R_H]}:=\Phi\inv([\mathcal R_H])$. Theorem \ref{thm:luzinsuslin} implies that $\widetilde{[H]_D}$ is Borel if and only if $[H]_D$ is, and that $\widetilde{[\mathcal R_H]}$ is Borel if and only if $[\mathcal R_H]$ is.

If we identify $G$ with the Borel subset of constant maps in $\LL^0(X,\mu,G)$, we have that $H=G\cap\widetilde{[H]}_D=G\cap \widetilde{[\mathcal R_H]}$ and so \textit{2}$\impl$\textit{1} and \textit{3}$\impl$\textit{1}. 
For the converse, first note that $\LL^0(X,\mu,H)$ is a Borel subset of $\LL^0(X,\mu,G)$, as shown in the corollary after Proposition 8 in \cite{MR0414775}. By Corollary \ref{cor:ctblerange}, $\LL^0_D(X,\mu,G)$ is also a Borel subset of $\LL^0(X,\mu,G)$ and the implications \textit{1}$\impl$\textit{2} and \textit{1}$\impl$\textit{3} hold because
\begin{align*}
\widetilde{[\mathcal R_H]}&=\widetilde{[\mathcal R_G]}\cap \LL^0(X,\mu,H)\\
\widetilde{[H]_D}&=\widetilde{[\mathcal R_H]}\cap \LL^0_D(X,\mu,G).\qedhere
\end{align*}
\end{proof}

Let us now point out two important consequences of the previous theorem. The first one is straightforward.

\begin{cor}\label{cor:boreldye}
Suppose that $G$ is a Polish group acting essentially freely and in a measure preserving Borel way on $(X,\mu)$. Then $[G]_D$ is Borel.
\end{cor}

\begin{qu}Can one remove the freeness assumption from the previous corollary?
\end{qu}

\begin{cor}\label{cor:nopol}There exists non Polishable ergodic full groups.
\end{cor}
\begin{proof}
Consider the free action of the circle $\mathbb S^1$ onto itself by translation. Let $H\leq \mathbb S^1$ be a non Borel subgroup which still acts ergodically. Then by Theorem \ref{thm:borelfullgroups} both $[H]_D$ and $[\mathcal R_H]$ will be non Borel full groups, and so by Corollary \ref{cor:ergopolishborel}, they cannot have a Polish group topology.

The existence of such an $H$ is a well-known consequence of the axiom of choice. Consider $\R$ as a $\Q$-vector space, and let $\tilde H$ be a hyperplane containing $\Q$. The subgroup $H:=\tilde H/\Z\leq \R/\Z=\mathbb S^1$ is a proper subgroup of $\mathbb S^1$ with countable infinite index, which acts ergodically because it contains an irrational. Such a subgroup $H$ can not be Lebesgue-measurable. Indeed, since $\mathbb S^1$ is covered by countably many translates of $H$, $H$ has non zero-measure. But then, $H$ must have finite-index, a contradiction.
\end{proof}

\section{Character rigidity for full groups}
In this section, we use a result of Dudko concerning characters of the full group of the tail equivalence relation $\mathcal R_0$ to classify characters of ergodic full groups admitting a Polish group topology. Recall that $\mathcal R_0$ is the countable pmp equivalence relation on $X=\{0,1\}^\N$ equipped with the product measure $\mu=(1/2\delta_0+1/2\delta_1)^{\otimes\N}$, defined by 
$$(x_n)\RR_0(y_n)\iff \exists p\in\N\,|\, \forall n\geq p, x_n=y_n.$$

\begin{thm}\label{thm:descrichar} For an ergodic Polish full group $\mathbb G$ we have the following dichotomy:
\begin{enumerate}
\item either $\mathbb G$ is the full group of a countable pmp equivalence relation, and all its continuous characters are (possibly infinite) convex combinations of the trivial character $\chi_0\equiv1$ and the characters $\{\chi_k\}_{k\geq1}$ given by 
$$\chi_k(g):=\mu(\{x\in X: g\cdot x=x\})^k,$$ 
\item  or $\mathbb G$ does not have any non trivial continuous character. 
\end{enumerate}
\end{thm}

The result of Dudko that we use may be stated as follows.

\begin{thm}[\cite{MR2813475}]\label{thm:charahyp}
Let $\chi$ be a continuous character of the full group of the tail equivalence relation $\mathcal R_0$. Then there is a unique sequence non negative coefficients $(\alpha_k)_{k\geq 0}$ such that $\sum_{k=0}^{+\infty} \alpha_k=1$ and $\chi=\sum_{k=0}^{+\infty} \alpha_k\chi_k.$
\end{thm}

Our result will follow from the description of the uniformly continuous characters of any ergodic full group, which was also observed by Gaboriau and Medynets (private communication). 

\begin{prop} Let $\mathbb G$ be an ergodic full group and let $\chi$ be a character of $\mathbb G$, continuous for the uniform topology. Then there is a unique sequence of non negative coefficients $(\alpha_k)_{k\geq 0}$ such that
$\chi=\sum_{k=0}^{+\infty} \alpha_k\chi_k.$
\end{prop}
\begin{proof}
It is a standard fact that up to conjugating, we may assume that $[\mathcal R_0]\subset \mathbb G$ (for a proof in the general setting of ergodic full groups, see \cite[Thm. 2.19]{lemai2014}). 
Let $\chi$ be a character of $\mathbb G$ continuous for the uniform topology. By Theorem \ref{thm:charahyp}, there is a sequence of non negative coefficients $(\alpha_k)_{k\geq 0}$ such that $\chi(g)=\sum_{k=0}^{+\infty} \alpha_k\chi_{k}(g)$ for all $g\in[\mathcal R_0]$. By Rokhlin's Lemma, the set $\mathcal F$ of elements of finite order of $\mathbb G$ is uniformly dense in $\mathbb G$. Since $\mathbb G$ is ergodic, every element of $\mathcal F$ is conjugate inside $\mathbb G$ to an element of $[\mathcal R_0]$. By definition $\chi$ is conjugacy-invariant, so $\chi(g)=\sum_{k=0}^{+\infty} \alpha_k\chi_{k}(g)$ for all $g\in \mathcal F$ and hence for any $g\in\mathbb G$ by continuity.
\end{proof}

\begin{proof}[Proof of theorem \ref{thm:descrichar}]
Let $\mathbb G$ be an ergodic full group equipped with a Polish topology $\tau$, and let $\chi$ be a $\tau$-continuous non-trivial character of $\mathbb G$. By Theorem \ref{thm:atmostonepolishergo}, $\tau$ is weaker than the uniform topology. So the character $\chi$ is continuous for the uniform topology and, by the previous proposition, there exists a sequence of non negative coefficients $(\alpha_k)_{k\geq 0}$ such that
$\chi=\sum_{k=0}^{+\infty} \alpha_k\chi_k$.

For all $T\in\Aut(X,\mu)$ and all $k\geq 1$, we have $\chi_k(T)=(1-d_u(T,\mathrm{id}_X))^k$, so $\chi_k(T_n)\to 1$ if and only if $d_u(T_n,\mathrm{id}_X)\to 0$. Since $\chi$ is not the trivial character, we deduce that $\chi(T_n)\to 1$ if and only if $d_u(T_n,\mathrm{id}_X)\to 0$. When $T_n$ converges to $\mathrm{id}_X$ for $\tau$, we have $\chi(T_n)\to 1$, which implies that $T_n$ converges to $\mathrm{id}_X$ in the uniform topology. We deduce that $\tau$ refines the uniform topology, hence they are equal. Since $\tau$ is Polish, we can conclude from Proposition \ref{prop:polishfordu} that $\mathbb G$ is the full group of a countable pmp equivalence relation. 
\end{proof}

\section{Perspectives}
\subsection{Orbit full groups for locally compact unimodular groups}\label{perspulc}
A forthcoming paper \cite{lcfullgroups} will be devoted to the study of full groups associated to measure-preserving actions of Polish locally compact unimodular groups. The existence of discrete sections for such actions plays a crucial role in our work. 

The first result is a locally compact version of \cite[Thm. 5.7]{MR2311665}. 

\begin{thm}Let $G$ be a Polish locally compact unimodular group acting ergodically and essentially freely on a standard probability space $(X,\mu)$. Then the following assertions are equivalent.
\begin{enumerate}[(i)]
\item $G$ is amenable.
\item $[\mathcal R_G]$ is extremely amenable.
\end{enumerate}
\end{thm}

Our second result is devoted to the topological rank\footnote{Recall that the \textbf{topological rank} of a topological group is the minimal number of elements needed to generate a dense subgroup.} of such full groups. For these full groups the situation much tamer than in the discrete case, which was described in \cite{gentopergo}.
\begin{thm}Let $G$ be a non discrete Polish locally compact unimodular group acting ergodically and freely on a standard probability space $(X,\mu)$. Then $[\mathcal R_G]$ has  topological rank $2$.

\end{thm}

\subsection{Orbit full groups for non-archimedean Polish groups}

Let us consider the Polish group $ \mathfrak S_\infty$ of permutations of $\N$. The Bernoulli shift $\mathfrak S_\infty\act[0,1]^\N$ is an essentially free action, and it would be interesting to understand the orbit full groups given by the restriction of this action to closed subgroups of $\mathfrak S_\infty$. These groups are also called  non-archimedean Polish groups and they are exactly the  automorphism groups of countable structures (see \cite[Thm. 1.5.1]{MR1425877}). Some of their topological properties  can  be understood in terms of the model-theoretic properties of the corresponding countable structure (see e.g. \cite{MR2140630,MR2308230}). 

As a concrete example, the group $\Aut(\Q,<)$ of order preserving bijections of the rationals is extremely amenable, as a consequence of Ramsey's theorem \cite{MR1608494}. Is the orbit full group associated to the Bernoulli shift $\Aut(\Q,<)\act[0,1]^\Q$ extremely amenable?


\bibliographystyle{myalpha}

\end{document}